\documentclass[11pt,reqno]{actasmonteiro}

\usepackage{amssymb,latexsym}
\usepackage{url}
\usepackage{color}

\newtheorem{thm}{Theorem}

\newtheorem{lem}{Lemma}
\newtheorem{cor}{Corollary}

\begin{document}


\title[Continuity properties of strongly singular operators]{Continuity properties of strongly singular integral operators for extreme values of $p$}


\author{Fabio Berra}
\address{Facultad de Ingeniería Química (CONICET-UNL), Santa Fe, Argentina}
\email{fberra@santafe-conicet.gov.ar}


\author{Gladis Pradolini}
\address{Facultad de Ingeniería Química (CONICET-UNL), Santa Fe, Argentina}
\email{gladis.pradolini@gmail.com}


\author{Wilfredo Ramos}
\address{Instituto de Modelado e Innovación Tecnológica (CONICET-UNNE), Corrientes, Argentina}
\email{wilfredo.ramos@comunidad.unne.edu.ar}


\author{Ignacio Viltes}
\address{Facultad de Ingeniería Química (CONICET-UNL), Santa Fe, Argentina}
\email{viltesignacio@hotmail.com}

\thanks{\thanks{The authors were supported by Consejo Nacional de Investigaciones Científicas y Técnicas (CONICET), Universidad Nacional del Litoral and Gobierno de la Provincia de Santa Fe (Argentina)}}


\begin{abstract}
In this work, we establish continuity properties of strongly singular integral operators for extreme values of $p$. Particularly, weighted $L^\infty$-$BMO$ boundedness is obtained, generalizing Miyachi's result to the context of Muckenhoupt weights. As an application, we get an alternative proof of Chanillo's weighted $L^p$ estimates via extrapolation techniques.
\end{abstract}

\maketitle

\section{Introduction}
Given a positive real number $b$, we  consider the strongly singular integral operator
\[T_b f(x)=\operatorname{p.v.}\int_{\mathbb{R}^n}K_b (x-y)f(y)\,dy,\] 
where \[K_b(x)=\frac{e^{i|x|^{-b}}}{|x|^{n}}\chi_{\{|x|\leq 1\}}(x).\]
The operator above arises as convolution analogue of certain well-known singular Fourier multipliers (cf. \cite{Z02, W65, M81}). The expression \emph{strongly singular integral} reflects how the oscillatory component of $K_b$ compromises the regularity of the kernel near the diagonal $x=y$. This behavior gives rise to a generalized smoothness condition, expressed as
\[\int_{\{|y|>2|x|^{1/(1+b)}\}}|K_b(x-y)-K_b(y)|\,dy\leq C,\qquad 0<|x|\leq 1,\]
rendering classical Calderón–Zygmund techniques not applicable unless $b\equiv 0$. 
However, it is well known that $K_b\ast f$ satisfies strong-type $(p,p)$ inequalities for $1<p<\infty$ (\cite{H59, W65}), weak-type $(1,1)$ estimates (\cite{F70}) and it can be extended to a bounded operator from $L^{\infty}$ into $BMO$ \cite{FS72, M81}, sharing the same continuity properties of the aforementioned operators.

The weighted $A_p$ theory for strongly singular integral operators was developed by Sagun Chanillo (\cite{C84}), relying on a key tool that involves the auxiliary kernel
\[\tilde{K}_{b,p'}(x)=\frac{e^{i|x|^{-b}}}{|x|^{n(2+b)/p'}}\chi_{\{|x|\leq 1\}}(x),\]
and proving that  
\begin{equation} \label{lem: Lemma 2.1 (Chanillo)}
\|\tilde{K}_{b,p'}\ast f\|_{L^{p'}}\leq C\|f\|_{L^{p}}
\end{equation} 
for $1<p<\infty$ satisfying $(2+b)/p'<1$.
As far as we know, weighted boundedness problems remain open when $p=\infty$, since the classical results in \cite{FS72, M81} address the unweighted case only, as well as more recent extensions of these operators (\cite{DMI80, AM86}). 

In this work, we establish the following main result, which is a weighted version of the corresponding estimate obtained in \cite{M81}.

\begin{thm} \label{thm: Acotación L^infinto-BMO}
Let $b$ be a positive real number and let $w\in A_{1}$. Then there exists a positive constant $C$ such that the inequality 
\begin{equation} \label{eq: Acotación L^infinto-BMO}
\| \chi_Bw^{-1}\|_{L^{\infty}}\left(\frac{1}{|B|}\int_{B}|T_{b}f(x)-(T_{b}f)_B|\,dx\right)\leq C\| fw^{-1}\|_{L^{\infty}}
\end{equation}
holds for every ball $B\subseteq\mathbb{R}^n$ and every measurable function $f$ for which $fw^{-1}\in L^\infty$.
\end{thm}

Furthermore, by applying Theorem 3 in \cite{HMS88}, we obtain an alternative proof of the weighted estimates in \cite{C84} through extrapolation techniques, leading to the following corollary.

\begin{cor}
Let $b$ be a positive real number and let $w\in A_p$. Then there exists a positive constant $C$ such that the inequality
\[\|K_b\ast f\|_{L^p(w)}\leq C\|f\|_{L^p(w)}\]
holds for every function $f\in L^p(w)$.
\end{cor}

\section{Proofs of the main results}
Before proceeding to the proof of Theorem \ref{thm: Acotación L^infinto-BMO} we state and prove the following estimate, which will be useful for this purpose.
\begin{lem} \label{lem: Hörmander pesada}
Let $b$ be a positive real number and let $w$ be a weight. Then there exists a positive constant $C$ such that the inequality
\[\int_{\{|x_0-y|>2|x_0-x|^{1/(1+b)}\}}|K_b(x-y)-K_b(x_0-y)|\,w(y)\,dy\leq CMw(x)\]
holds for almost every $x\in\mathbb{R}^n\setminus\{x_0\}$ with $|x_0-x|\leq 1$.
\end{lem}

\begin{proof}
Since $|x_0-x|\leq 1$ and $b>0$, we get $|x_0-y|>2|x_0-x|^{1/(1+b)}\geq 2|x_0-x|$.\linebreak
Then, it is clear that
\begin{equation} \label{eq: Triangular auxiliar (lema Hörmander pesada)}
\frac{|x_0-y|}{2}<|x_0-y|-|x_0-x|\leq |x-y|\leq |x_0-y|+|x_0-x|<\frac{3}{2}|x_0-y|.
\end{equation}
Let us denote
\begin{align*}
I&=\int_{\{|x_0-y|>2|x_0-x|^{1/(1+b)}\}}|K_b(x-y)-K_b(x_0-y)|\,w(y)\,dy\\
&=\int_{\{|x_0-y|>2|x_0-x|^{1/(1+b)}\}}\left|\frac{e^{i|x-y|^{-b}}}{|x-y|^{n}}\chi_{\{|x-y|\leq 1\}}(x-y)-\frac{e^{i|x-x_0|^{-b}}}{|x-x_0|^{n}}\chi_{\{|x-x_0|\leq 1\}}(x-x_0)\right|\,w(y)\,dy.
\end{align*}
We shall consider three cases.

\noindent\emph{Case }1. $|x-y|\leq 1$ and $|x_0-y|>1$. From \eqref{eq: Triangular auxiliar (lema Hörmander pesada)} we obtain that $1/2<|x-y|\leq 1$ and, consequently,
\begin{align*}
I&=\int_{\{|x_0-y|>2|x_0-x|^{1/(1+b)}\}}\left|\frac{e^{i|x-y|^{-b}}}{|x-y|^{n}}\chi_{\{|x-y|\leq 1\}}(x-y)\right|w(y)\,dy\\
&\leq \int_{\{1/2<|x-y|\leq 1\}}\frac{w(y)}{|x-y|^{n}}\,dy\\
&\leq 2^{n}\int_{B(x,1)}w(y)\,dy\\
&\lesssim Mw(x).
\end{align*}

\noindent\emph{Case }2. $|x-y|>1$ y $|x_0-y|\leq 1$. Inequality \eqref{eq: Triangular auxiliar (lema Hörmander pesada)} now implies that $|x_0-y|^{-n}\leq (2/3)^{-n}$. Therefore, 
\begin{align*}
I&=\int_{\{|x_0-y|>2|x_0-x|^{1/(1+b)}\}}\left|\frac{e^{i|x_0-y|^{-b}}}{|x_0-y|^{n}}\chi_{\{|x_0-y|\leq 1\}}(x_0-y)\right|w(y)\,dy\\
&\leq \int_{\{1\geq |x_0-y|>2|x_0-x|^{1/(1+b)}\}}\frac{w(y)}{|x_0-y|^{n}}\,dy\\
&\leq \left(\frac{3}{2}\right)^{n}\int_{B(x_0,1)}w(y)\,dy\\
&\lesssim Mw(x).   
\end{align*}

\noindent\emph{Case }3. $|x-y|\leq 1$ and $|x_0-y|\leq 1$. An application of the mean value theorem yields
\[\left|\frac{e^{i|x-y|^{-b}}}{|x-y|^{n}}-\frac{e^{i|x-x_0|^{-b}}}{|x-x_0|^{n}}\right|=\left|-e^{i\xi^{-b}}\left(\frac{n}{\xi^{n+1}}+\frac{ib}{\xi^{n+(1+b)}}\right)\right|||x-y|-|x_0-y||\lesssim \frac{|x_0-x|}{|x_0-y|^{n+(1+b)}},\]
where $\xi>0$ is a number between $|x-y|$ and $|x_0-y|$. Note that the last estimate follows from \eqref{eq: Triangular auxiliar (lema Hörmander pesada)} combined with the hypothesis. Let $\tilde{B}=B(x_0,|x-x_0|^{1/(1+b)})$. A standard covering argument provides
\begin{align*}
I&\lesssim |x_0-x|\int_{\{|x_0-y|>2|x_0-x|^{1/(1+b)}\}}\frac{w(y)}{|x_0-y|^{n+(1+b)}}\,dy\\
&\lesssim |x_0-x|\sum_{k=1}^{\infty}\int_{\{|x_0-y|\sim 2^{k}|x_0-x|^{1/(1+b)}\}}\frac{w(y)}{|x_0-y|^{n+(1+b)}}\,dy\\
&\lesssim |x_0-x|\sum_{k=1}^{\infty}\frac{2^{-k(1+b)}}{|x_0-x|}\left(\frac{1}{|2^{k+1}\tilde{B}|}\int_{2^{k+1}\tilde{B}}w(y)\,dy\right)\\
&\lesssim Mw(x).\qedhere
\end{align*}
\end{proof} 

\medskip

\begin{proof}[Proof of Theorem~\ref{thm: Acotación L^infinto-BMO}]
Since $w\in A_{1}$, there exists $1<p_0<\infty$ for which the reverse Hölder inequality
\begin{equation} \label{eq: Reverse Hölder (resultado principal)}
\left(\frac{1}{|B|}\int_B  w^p(x)\,dx\right)^{1/p}\leq \frac{C}{|B|}\int_B w(x)\,dx,
\end{equation}
holds for every ball $B\subseteq \mathbb{R}^n$ and every $1<p<p_0$. Furthermore, we can choose the exponent $p$ such that $(2+b)/p'<1$. If $fw^{-1}\in L^{\infty}$, it follows that $f\chi_K \in L^p$ whenever $K\subseteq \mathbb{R}^n$ is a compact set. 

Fixed any ball $B=B(x_0, r)\subseteq \mathbb{R}^n$, we show that there exists a positive constant $\lambda$ such that the following inequality holds
\begin{equation} \label{eq: Acotación L^infinto-BMO (2)}
\frac{1}{|B|}\int_{B}|K_{b}\ast f(x)-\lambda|\,dx\leq \left(\frac{C}{|B|}\int_B w(x)\,dx\right)\|fw^{-1}\|_{L^{\infty}}.
\end{equation}
We shall consider two cases.

\emph{Case }1. Let us assume $0<r\leq 1$ and define $\tilde{B}=B(x_0,r^{1/(1+b)})$. Let $f=\sum_{j=1}^{3} f_j(x)$, where
\[f_1=f\chi_{2B},\quad f_2=f\chi_{2\tilde{B}\setminus 2B},\quad\text{and}\quad f_3=f\chi_{\mathbb{R}^n \setminus 2\tilde{B}}.\]
Since $K_b\ast f$ is a linear operator, it suffices to verify \eqref{eq: Acotación L^infinto-BMO (2)} separately for each $f_j$, that is
\[|K_b\ast f(x)-\lambda|\leq |K_b\ast  f_1(x)|+|K_b\ast f_2(x)|+|K_b\ast f_3(x)-\lambda|.\]

We first deal with $f_1$. By applying Jensen’s inequality and the $L^p$-boundedness of $K_b\ast f$ for $1<p<\infty$, 
\begin{align}
\frac{1}{|B|}\int_{B}|K_b\ast f_1(x)|\,dx&\lesssim \left(\frac{1}{|B|}\int_{B}|K_{b}\ast f_1(x)|^{p}\,dx\right)^{1/p}\nonumber\\
&\lesssim \left(\frac{1}{|B|}\int_{\mathbb{R}^{n}}|f_{1}(x)|^{p}\,dx\right)^{1/p}\lesssim \left(\frac{1}{|2B|}\int_{2B}w^p(x)\,dx\right)^{1/p}\|fw^{-1}\|_{L^{\infty}}\nonumber\\
&\lesssim \left(\frac{1}{|2B|}\int_{2B}w(x)\,dx\right)\|fw^{-1}\|_{L^{\infty}},
\label{eq: Estimación f1 (resultado principal)}
\end{align}
by virtue of \eqref{eq: Reverse Hölder (resultado principal)}. Since $B\subseteq 2B$, the $A_1$ condition applied to $w$ leads to
\begin{equation} \label{eq: Condición A1 (resultado principal)}
\frac{1}{|2B|}\int_{2B}w(x)\,dx\lesssim \inf_{x\in 2B}w(x)\lesssim \inf_{x\in B}w(x)\lesssim \frac{1}{|B|}\int_{B}w(x)\,dx.
\end{equation}
Combining \eqref{eq: Estimación f1 (resultado principal)} and \eqref{eq: Condición A1 (resultado principal)}, we conclude the estimate for $f_1$.

\medskip

On the other hand, if $x\in B$ we get
\[K_{b}\ast f_{2}(x)=\int_{\mathbb{R}^{n}}\frac{e^{i|x-y|^{-b'}}}{|x-y|^{n}}\chi_{\{|x-y|\leq 1\}}(x-y)f_{2}(y)\,dy.\]
We split this expression into two terms as follows
{\small\begin{align*}
K_{b}\ast f_{2}(x)&=\int_{\mathbb{R}^n}\frac{e^{i|x-y|^{-b'}}}{|x-y|^{n(2+b)/p'}}\chi_{\{|x-y|\leq 1\}}(x-y)\left(\frac{1}{|x-y|^{n-n(2+b)/p'}}-\frac{1}{|x_{0}-y|^{n-n(2+b)/p'}}\right)f_{2}(y)\,dy\\
&\qquad +\int_{\mathbb{R}^n}\frac{e^{i|x-y|^{-b'}}}{|x-y|^{n(2+b)/p'}}\chi_{\{|x-y|\leq 1\}}(x-y)\left(\frac{f_{2}(y)}{|x_{0}-y|^{n-n(2+b)/p'}}\right)\,dy\\
&=A(x)+B(x).  
\end{align*}}
To majorize $|A(x)|$, note that $|x_{0}-y|\geq 2r> 2|x_{0}-x|$. By the  mean value theorem we obtain that
\[\left|\frac{1}{|x-y|^{n-n(2+b)/p'}}-\frac{1}{|x_{0}-y|^{n-n(2+b)/p'}}\right|\lesssim \frac{|x_{0}-x|}{|x_{0}-y|^{n-n(2+b)/p'+1}}.\]
By decomposing adequately the domain in the integral, we get that
\[|A(x)|\lesssim|x_{0}-x|\int_{\{|x_{0}-y|\geq 2r\}}\frac{|f(y)|}{|x_{0}-y|^{n+1}}\,dy\lesssim r\sum_{k=1}^{\infty}\frac{2^{-k}}{r}\left(\frac{1}{|2^{k+1}B|}\int_{2^{k+1}B}w(y)\,dy\right)\|fw^{-1}\|_{L^{\infty}}.\]
The desired estimate for $A(x)$ then follows similarly as in \eqref{eq: Condición A1 (resultado principal)} by the $A_{1}$-condition of $w$.

\medskip

Regarding the remaining term, observe that 
\[B(x)=\tilde{K}_{b,p'}\ast\left(\frac{f_2(\cdot)}{|x_0-\cdot|^{n-n(2+b)/p'}}\right)(x).\]
The condition $(2+b)/p' < 1$ enables the application of \eqref{lem: Lemma 2.1 (Chanillo)}, providing
{\small\[\left(\frac{1}{|B|}\int_B \left|\tilde{K}_{b,p'}\ast\left(\frac{f_2(\cdot)}{|x_0-\cdot|^{n-n(2+b)/p'}}\right)(x)\right|^{p'}\,dx\right)^{1/p'}\lesssim \frac{1}{|B|^{1/p'}}\left(\int_{\mathbb{R}^n}\frac{|f_2(x)|^{p}}{|x_0-x|^{n-n(p-1)(1+b)}}\, dx\right)^{1/p}.\]}
Select an integer $k_0$ such that $2^{k_0}r<2r^{1/(1+b)}\leq 2^{k_0+1}r$. Applying \eqref{eq: Reverse Hölder (resultado principal)} and \eqref{eq: Condición A1 (resultado principal)}, the integral above is bounded by
\begin{align*}
&\frac{1}{|B|^{1/p'}}\left(\sum_{k=1}^{k_0}|2^kB|^{(p-1)(1+b)}\frac{1}{|2^{k+1}B|}\int_{2^{k+1}B}w^p(x)\,dx\right)^{1/p}\|fw^{-1}\|_{L^{\infty}}\\
&\qquad\lesssim\frac{1}{|B|^{1/p'}}\left(\sum_{k=1}^{k_0}|2^kB|^{(p-1)(1+b)}\right)^{1/p}\left(\frac{1}{|B|}\int_{B}w(x)\,dx\right)\|fw^{-1}\|_{L^{\infty}}.
\end{align*}
Expanding the geometric sum, 
\begin{align*}
\frac{1}{|B|^{1/p'}}\left(\sum_{k=1}^{k_0}|2^kB|^{(p-1)(1+b)}\right)^{1/p}&=|B|^{b/p'}\left(\frac{(2^{n(p-1)(1+b)})^{k_0+1}-2^{n(p-1)(1+b)}}{2^{n(p-1)(1+b)}-1}\right)^{1/p}\\
&\lesssim |B|^{b/p'}(2^{k_0+1})^{n(1+b)/p'}\lesssim|B|^{b/p'}r^{-nb/p'}\approx 1,
\end{align*}
since $2^{k_0+1}\lesssim r^{-b/(1+b)}$.

\medskip

Finally, we deal with $f_3$. Since $f_3$ is supported away from $x_0$, by taking $\lambda=K_b\ast f_3(x_0)$, it follows that
\[|K_b\ast f_3(x_0)|\lesssim\left(\int_{\mathbb{R}^n\setminus 2\tilde{B}}\frac{w(x)}{|x_0-x|^n}\chi_{\{|x_0-x|\leq 1\}}(x)\,dx\right)\|fw^{-1}\|_{L^{\infty}}<\infty.\]
By Lemma~\ref{lem: Hörmander pesada}, since $w\in A_1$, we obtain
\begin{align*}
|K_b\ast f_3(x)-\lambda|&\lesssim \left(\int_{\{|x_0-y|>2|x_0-x|^{1/(1+b)}\}}|K_b(x-y)-K_b(x_0-y)|w(y)\,dy\right)\|fw^{-1}\|_{L^{\infty}}\\
&\lesssim Mw(x)\|fw^{-1}\|_{L^{\infty}}\\
&\lesssim w(x)\|fw^{-1}\|_{L^{\infty}},\qquad
\end{align*}
for a.e. $x\in B$, which implies the desired estimate for this case.

\emph{Case }2. Suppose that $r>1$. In this case, we write $f=f_1+f_2$, where $f_1=f\chi_{2B}$ and $f_2=f\chi_{\mathbb{R}^n\setminus 2B}$. Observe that the estimate for $f_1$ can be performed as in the previous case, since does not depend on the value of $r$. Thus, it only remains to verify the desired inequality for $f_2$. By definition, we obtain
\begin{align*}
\frac{1}{|B|}\int_B|K_b\ast f_2(x)-(K_b\ast f_{2})_B|\,dx&\lesssim\frac{1}{|B|}\int_B\int_{\mathbb{R}^n\setminus 2B}|K_b(x-y)||f(y)|\,dy\,dx\\
&\lesssim\frac{1}{|B|}\int_{B}\int_{\mathbb{R}^n\setminus 2B}\frac{|f(y)|}{|x-y|^{n}}\chi_{\{|x-y|\leq 1\}}\,dy\,dx.
\end{align*}
However, for $x\in B$ and $y\in\mathbb{R}^n\setminus 2B$, we have $|x-y|\geq 2r>2$, so the inner integrand vanishes. Therefore,
\[\frac{1}{|B|}\int_B|K_b\ast f_2(x)-(K_b\ast f_{2})_B|\,dx=0.\qedhere\]
\end{proof}





\end{document}